\documentclass[10pt]{amsart}
\usepackage{amsthm,amsmath,amssymb,amsfonts}
\usepackage[american]{babel}
\usepackage{cite}
\usepackage{url}

\allowdisplaybreaks

\newtheorem{theorem}{Theorem}[section] 
\newtheorem{lemma}[theorem]{Lemma}
\newtheorem{proposition}[theorem]{Proposition}

\theoremstyle{definition}
\newtheorem{example}[theorem]{Example}

\numberwithin{equation}{section}

\begin{document}


\newcommand{\abs}[1]{\lvert#1\rvert}
\newcommand{\NU}[1]{\text{\rm V}(#1)}
\newcommand{\lara}[1]{\langle{#1}\rangle}
\newcommand{\ZZ}{\mathbb{Z}}
\newcommand{\QQ}{\mathbb{Q}}
\newcommand{\RR}{\mathbb{R}}
\newcommand{\sRR}{{\scriptstyle\mathbb{R}}}
\newcommand{\NN}{\mbox{$\mathbb{N}$}}
\newcommand{\CC}{\mbox{$\mathbb{C}$}}
\newcommand{\paug}[2]{\varepsilon_{#1}(#2)}
\newcommand{\cclr}[1]{\text{\rm ccl}_{r}(#1)}


\title[Finite groups of units in 
\protect{\mbox{$\ZZ[\text{\rm PSL}(2,\lowercase{q})]$}}]
{Finite groups of units and their composition factors in the integral
group rings of the groups $\boldsymbol{\text{\rm\bf PSL}(2,q)}$}

\author{Martin Hertweck}
\address{Universit\"at Stuttgart, Fachbereich Mathematik, IGT,
Pfaffenwald\-ring 57, 70550 Stuttgart, Germany}
\email{hertweck@mathematik.uni-stuttgart.de}

\author{Christian R. H\"ofert}
\address{Universit\"at Stuttgart, Fachbereich Mathematik, IGT,
Pfaffenwald\-ring 57, 70550 Stuttgart, Germany}
\email{hoefert@mathematik.uni-stuttgart.de}

\author{Wolfgang Kimmerle}
\address{Universit\"at Stuttgart, Fachbereich Mathematik, IGT,
Pfaffenwald\-ring 57, 70550 Stuttgart, Germany}
\email{kimmerle@mathematik.uni-stuttgart.de}

\date{\today}


\begin{abstract}
Let $G$ denote the projective special linear group $\text{\rm PSL}(2,q)$, 
for a prime power $q$. It is shown that a  finite $2$-subgroup of the
group $\NU{\ZZ G}$ of augmentation $1$ units in the integral group ring
$\ZZ G$ of $G$ is isomorphic to a subgroup of $G$. Furthermore, it is 
shown that a composition factor of a finite subgroup of $\NU{\ZZ G}$ is 
isomorphic to a subgroup of $G$.
\end{abstract}

\maketitle

\section{Introduction}\label{Sec:intro}

A conjecture of H.~Zassenhaus from the 1970s asserts that
for a finite group $G$, every torsion unit in its integral group ring
$\ZZ G$ is conjugate to an element of $\pm G$ by a unit in the 
rational group ring $\QQ G$. For known results on this still unsolved
conjecture the reader is referred to
\cite[Chapter~5]{Seh:93}, \cite[\S\,8]{Seh:03} and \cite{Her:05,Her:06}.
In fact this conjugacy question makes sense even for finite groups of
units in $\ZZ G$.
The outstanding result in the field is Weiss's proof \cite{Wei:91}
that for nilpotent $G$, this strong version of the conjecture is true.

The question begs to be asked though, is a finite group of units in 
$\ZZ G$ (for general $G$) necessarily isomorphic to a subgroup of $\pm G$?
Related issues are addressed in Problems No.~19 and 20 from \cite{OWR:07}.
We present results from the second 
authors Ph.D.~thesis \cite{Hoef:08} when $G$ is a two-dimensional 
projective special linear group $\text{\rm PSL}(2,q)$, $q$ a prime power. 
We remark that no feasible approach is currently available for obtaining
some general results which does not boil down to conjugacy questions.

It is always enough to consider only finite subgroups of $\NU{\ZZ G}$,
the group of augmentation $1$ units in $\ZZ G$. 
Focus may be either on particular classes of groups $G$, or on 
particular classes (even certain isomorphism types) of finite subgroups of
$\NU{\ZZ G}$ (for general $G$), or on both.
To indicate the difficulty of the terrain: a torsion
unit in $\NU{\ZZ G}$ is not known to have the same order as some element
of $G$, except when it is of prime power order \cite{CoLi:65} or $G$ 
is solvable \cite{Her:07b}. The first result, already known for a long
time, shows that the exponent of a finite subgroup of $\NU{\ZZ G}$ 
divides the exponent of $G$. Another general result is that 
the order of a finite subgroup of $\NU{\ZZ G}$ divides the order of $G$
(Berman; cf.\ \cite{San:81}, for proof see \cite[Lemma~(37.3)]{Seh:93}).
More recently, it has been shown \cite{Her:07a} that a finite subgroup 
of $\NU{\ZZ G}$ has cyclic Sylow $p$-subgroups provided this is true 
for $G$, as yet another application of partial augmentations. 
This happened after it was observed \cite{Ki:07}
that $\NU{\ZZ G}$ 
contains no four-group provided there is no such in $G$, by the 
Berman--Higman result on the vanishing of $1$-coefficients of torsion
units and by a deep theorem of Brauer and Suzuki on groups with
quaternion Sylow $2$-subgroups.

The Zassenhaus conjecture for torsion units in 
$\ZZ[\text{\rm PSL}(2,q)]$ (brackets are inserted for better readability)
has been studied in \cite[\S\,6]{Her:05b} using a modular version of the
Luthar--Passi method.

In \S\,\ref{sec:2groups}, 
it is shown that finite $2$-subgroups of $\NU{\ZZ[\text{\rm PSL}(2,q)]}$
are isomorphic to subgroups of $\text{\rm PSL}(2,q)$ (only the case $q$
odd matters, when Sylow $2$-subgroups of $\text{\rm PSL}(2,q)$ are dihedral).
The method of proof is that of \cite{Her:07a}, 
and we sketch the relevant idea. 
Let the finite group $H$ be a (putative) subgroup of $\NU{\ZZ G}$.
Let $\xi$ be an ordinary character of the group $G$. When $\xi$ is
viewed as a trace function on the complex group ring $\CC G$ of $G$,
then its restriction $\xi_{H}$ to $H$ is a character of $H$. Suppose 
that the values $\xi(x)$, $x\in H$, are sufficiently well known, which
essentially means that some knowledge is available about the Zassenhaus 
conjecture for torsion units in $\NU{\ZZ G}$ of the same order as some
element of $H$. Then, for an irreducible character $\lambda$ of $H$,
the scalar product $\lara{\lambda,\xi_{H}}$, defined as 
$\frac{1}{\abs{H}}\sum_{x\in H}\lambda(x)\xi_{H}(x^{-1})$, might be
calculated accurately enough to show that it is not a
nonnegative rational integer, a contradiction showing that $H$ is not
a subgroup of $\NU{\ZZ G}$ since $\lara{\lambda,\xi_{H}}$ is the number
of times $\lambda$ occurs in $\xi_{H}$. This interpretation of
$\lara{\lambda,\xi_{H}}$ suggests that it might also be useful 
to know over which fields a representation corresponding to $\xi$
can be realized (keep $\RR$, the real numbers, as a first choice in mind).

Let $H$ be a finite subgroup of $\NU{\ZZ G}$.
When $G$ is solvable, then so is $H$ (see \cite[Lemma~(7.4)]{Seh:93}).
When $G$ is nonsolvable, one might ask whether the (nonabelian) 
composition factors of $H$ are isomorphic to subquotients of $G$
(Problem~20 from \cite{OWR:07}). This holds if $H$ is a group basis 
of $\ZZ G$ (meaning that $\abs{H}=\abs{G}$), when it even has the same
chief factors as $G$ including multiplicities \cite{KiLySaTe:90}.

In \S\,\ref{sec:CF}, it is shown that for a finite subgroup of 
$\NU{\ZZ[\text{\rm PSL}(2,q)]}$, its composition factors 
are isomorphic to subgroups of $\text{\rm PSL}(2,q)$. 
This was shown in \cite{Hoef:08} under the additional assumption that
$\text{\rm PSL}(2,q)$ has elementary abelian Sylow $2$-subgroups.
Our approach differs slightly in that \cite[Lemma~3.7]{Hoef:08}
is replaced by the obvious generalization, Lemma~\ref{lem2} below,
and it is subsequently used that a finite subgroup of 
$\NU{\ZZ[\text{\rm PSL}(2,q)]}$ has abelian or dihedral Sylow 
$2$-subgroups (and the part of the classification of the finite simple
groups concerning such groups).
We remark that Dickson has given a complete list of the subgroups
of $\text{\rm SL}(2,q)$, see Chapter~3, \S\,6 in \cite{Suz:82}, for 
example. The nonabelian simple subgroups of $\text{\rm PSL}(2,q)$ are
isomorphic to $\text{\rm PSL}(2,p^{m})$, with the field of $p^{m}$ 
elements a subfield of the field of $q$ elements and characteristic $p$,
or isomorphic to the alternating group
$A_{5}$ if $5$ divides the group order.

\section{Finite $2$-groups of units in $\ZZ[\text{\rm PSL}(2,q)]$}%
\label{sec:2groups}

Let $G=\text{\rm PSL}(2,q)$, where $q$ is a prime power. 
We will show that finite $2$-subgroups of $\NU{\ZZ G}$ are
isomorphic to subgroups of $G$. The Sylow $2$-subgroups of $G$ are
elementary abelian if $q$ is even and dihedral groups if $q$ is odd
(see \cite[2.8.3]{Gor:68}). 
Remember that the order of a finite subgroup of $\NU{\ZZ G}$ divides
the order of $G$.

When $q$ is even, a Sylow $2$-subgroup of $G$ has exponent $2$ and so has 
any finite $2$-subgroup of $\NU{\ZZ G}$ by \cite[Corollary~4.1]{CoLi:65}.
Then it follows that a finite $2$-subgroup of $\NU{\ZZ G}$ is elementary
abelian, and therefore isomorphic to a subgroup of $G$.

So we only have to deal with the case $q$ odd. 
Note that a dihedral $2$-group contains, for each divisor $n$ of its
order, $n\geq 4$, a cyclic and a dihedral subgroup of order $n$.
Thus it suffices to prove the following theorem.

\begin{theorem}\label{t2UG}
Let $H$ be a finite $2$-subgroup of $\NU{\ZZ G}$, where
$G=\text{\rm PSL}(2,q)$ with $q$ an odd prime power.
Then $H$ is either cyclic or a dihedral group.
\end{theorem}
\begin{proof}
Let $\varepsilon\in\{-1,1\}$ such that $q\equiv\varepsilon\mod{4}$.
For convenience of the reader, the (ordinary) character table of 
$G$ is shown in Table~\ref{Table3}
(in the notation from \cite[\S\,38]{Dor:71}).

\newlength{\uni}
\settowidth{\uni}{$- $}
\newlength{\unii}
\settowidth{\unii}{$\varepsilon=-1:\; {}$}
\begin{table}[h] 
\[ \begin{array}{c}
\begin{array}{c|cccccc} \hline
\text{class of} & 1 & c & d & a^{l} & b^{m} \rule[-7pt]{0pt}{20pt} \\ \hline
\text{order} 
& 1 & p & p & 
\text{\rm of $a$ is }
\frac{q-1}{2} & \text{\rm of $b$ is }\frac{q+1}{2}
\rule[-7pt]{0pt}{20pt} \\ \hline
1 & 1 & 1 & 1 & 1 & 1 \rule[0pt]{0pt}{13pt} \\
\psi & q & 0 & 0 & 1 & -1\hspace*{\uni} \\
\chi_{i} & q+1 & 1 & 1 & \rho^{il}+ \rho^{-il} & 0 \\
\theta_{j} & q-1 & -1\hspace*{\uni} & -1\hspace*{\uni} 
& 0 & -(\sigma^{jm}+ \sigma^{-jm}) \\
\eta_{1} & \frac{1}{2}(q+\varepsilon) & 
\frac{1}{2}(\varepsilon+\sqrt{\varepsilon q})
 & \frac{1}{2}(\varepsilon-\sqrt{\varepsilon q})
 & (-1)^{l}\delta_{\varepsilon,1} & (-1)^{m+1}\delta_{\varepsilon,-1} 
\rule[-7pt]{0pt}{20pt} \\
\eta_{2} & \frac{1}{2}(q+\varepsilon) & 
\frac{1}{2}(\varepsilon-\sqrt{\varepsilon q})
 & \frac{1}{2}(\varepsilon+\sqrt{\varepsilon q})
& (-1)^{l}\delta_{\varepsilon,1} & (-1)^{m+1}\delta_{\varepsilon,-1} 
\rule[-7pt]{0pt}{5pt} \\ \hline
\end{array} \\
\begin{array}{rl}
\text{Entries:}\ 
& \delta_{\varepsilon,\pm 1} \text{ Kronecker symbol,} \\ 
& \rho=e^{4\pi i/(q-1)},\;\sigma=e^{4\pi i/(q+1)}, \\
& \parbox[t]{\unii}{$\varepsilon=1:\; $ \hfill} 
1\leq i\leq \frac{1}{4}(q-5),\; 1\leq j,l,m\leq \frac{1}{4}(q-1),
\rule[6pt]{0pt}{5pt} \\
& \varepsilon=-1:\;
1\leq i,j,l\leq \frac{1}{4}(q-3),\; 1\leq m\leq \frac{1}{4}(q+1). 
\rule[6pt]{0pt}{5pt} \\
\end{array}
\rule[0pt]{0pt}{50pt}
\end{array} \] 
\rule[0pt]{0pt}{5pt}
\caption{Character table of $\text{\rm PSL}(2,q)$, 
$q=p^{f}\geq 5$, odd prime $p$}
\label{Table3}
\vspace*{-10pt}
\end{table}

We can assume that $\abs{H}\geq 8$. Then $8$ divides $\abs{G}$.
Since $\abs{G}=(q-1)q(q+1)/2$, the maximal power of $2$ dividing 
$\abs{G}$ divides $q-\varepsilon$. Hence $q-\varepsilon\equiv 0\mod{8}$.

We can assume that $q\neq 5$, since otherwise $G$ is isomorphic to
the alternating group $A_{5}$, and the statement of the theorem is
known (see \cite{DoJuMi:97}). Then we can set $\xi=\chi_{1}$ if 
$\varepsilon=1$ and $\xi=\theta_{1}$ if $\varepsilon=-1$. Note that 
$\xi(1)=q+\varepsilon$. The group $G$ has only one conjugacy class of
involutions. Let $s$ be an involution in $G$.
By \cite[Corollary~3.5]{Her:06}, it follows that an 
involution in $H$ is conjugate to $s$ by a unit
in $\QQ G$. So $\xi(x)=\xi(s)$ for an involution $x$ of $H$.
Note that $\xi(s)=-2\varepsilon$. Suppose that $H$ has an
element $u$ of order $4$. For an element $g$ of $G$ of order $4$ we have
$\xi(g)=0$. By \cite[Proposition~3.1]{Her:06}, it follows that
$\xi(u)=\paug{s}{u}\xi(s)$, where $\paug{s}{u}$ denotes the partial 
augmentation of $u$ at the conjugacy class of $s$. 
By \cite[Theorem~4.1]{CoLi:65}, $\paug{s}{u}$ is divisible by $2$ 
and so $\xi(u)\equiv 0\mod{4}$. Also note that $\xi(u)=\xi(u^{-1})$
since $\xi(u)$ is a rational integer. So 
$\xi(u)+\xi(u^{-1})\equiv 0\mod{8}$.

Suppose that $H$ is elementary abelian of order $8$. Let $\lambda$ be
an irreducible character of $H$ which is not principal. Then 
\[ \sum_{x\in H}\lambda(x)\xi(x^{-1})=\xi(1)+\xi(s)
\sum_{1\neq x\in H}\lambda(x)=(q+\varepsilon)-2\varepsilon(3-4)
=q+3\varepsilon. \]
Since $q+3\varepsilon\not\equiv 0\mod{8}$, this contradicts the fact
that $\lara{\lambda,\xi_{H}}$ is a (nonnegative) integer.

Now suppose that $H$ is the direct product of a cyclic group of order $4$
and a cyclic group of order $2$. Let $u,v\in H$ such that 
$\lara{u}$ and $\lara{v}$ are the two subgroups of order $4$ in $H$.
Let $\lambda$ be the principal character of $H$. Then
\[ \sum_{x\in H}\lambda(x)\xi(x^{-1})=(q+\varepsilon)+
3(-2\varepsilon)+2\xi(u)+2\xi(v)\equiv q+3\varepsilon\mod{8}, \]
again a contradiction.

We have shown that a maximal abelian subgroup of $H$ is either cyclic
or isomorphic to a four-group $V$ (the direct product of two cyclic 
groups of order $2$). Suppose that $H$ has a 
noncyclic abelian normal subgroup $N$. Then $N\cong V$, and $N$ is a
maximal abelian normal subgroup of $H$, so that the quotient $H/N$ acts
faithfully on $N$. It follows that $H$ is either $N$ or a dihedral
group of order $8$. Thus we can assume that $H$ has no 
noncyclic abelian normal subgroups. When $H$ is not cyclic or a
dihedral group, it then must be a semidihedral group or a 
(generalized) quaternion 
group (see \cite[5.4.10]{Gor:68}). A semidihedral group contains a
direct product of a cyclic group of order $4$ and a cyclic group of 
order $2$, so we have already ruled out the possibility of $H$ being
semidihedral. A (generalized) quaternion group has a quaternion group
of order $8$ as a subgroup. Hence the proof will be finished once
we have shown that $H$ cannot be the quaternion group of order $8$.

Assume the contrary. 
Let $\lambda$ be the irreducible character of $H$ of degree $2$.
Then $\lambda(z)=-2$ for the involution $z$ in $H$ and 
$\lambda(u)=0$ for an element $u$ of order $4$ in $H$.
It follows that
\[ \sum_{x\in H}\lambda(x)\xi(x^{-1})=2(q+\varepsilon)
+(-2)(-2\varepsilon)=2(q-\varepsilon)+8\varepsilon, \]
so $\lara{\lambda,\xi_{H}}$ is an odd integer.
This reminds us of the following fact. If $W$ is an irreducible 
$\RR H$-module such that $\CC\otimes_{\sRR}W$ has character
$\mu$ satisfying $\lara{\lambda,\mu}\neq 0$, then $W$ is the quaternion
algebra on which $H$ acts by multiplication, and so 
$\lara{\lambda,\mu}=2$. But a $\CC G$-module $M$ with character $\xi$ can
be realized over $\RR$, which means that there exists an isomorphism
$M\cong\CC\otimes_{\sRR}M_{0}$ for some $\RR G$-module $M_{0}$.
This can be shown by calculating the Frobenius--Schur indicator of $M$ 
in terms of the character $\xi$ (see \cite[XI.8.3]{HuBl:82}). This 
implies that $\lara{\lambda,\xi_{H}}$ is an even integer, a contradiction.
\end{proof}

In view of the final contradiction in the above proof, we remark that
the Schur indices, over the rational field, of the irreducible 
characters of $\text{SL}(2,q)$ and the simple direct summands of the
rational group algebra $\QQ[\text{SL}(2,q)]$ are known 
\cite{Jan:74}, \cite{Sha:83}. In particular, the Schur indices of the 
irreducible characters of $\text{PSL}(2,q)$ are all $1$.

We give an instance where the theorem can be applied.
The group $\text{\rm PSL}(2,7)$, of order $168$, is the second smallest
nonabelian simple group. In \cite[Example~3.6]{Her:06}, it has been
shown that for $G=\text{\rm PSL}(2,7)$, the (first) Zassenhaus conjecture 
is valid, that is, any torsion unit in $\NU{\ZZ G}$ is conjugate to an
element of $G$ by a unit in $\QQ G$.

\begin{example}
Let $G=\text{\rm PSL}(2,7)$. We show that a finite subgroup $H$ of
$\NU{\ZZ G}$ is conjugate to a subgroup of $G$ by a unit in $\QQ G$.
By \cite[Theorem~1]{BlHiKi:95}, any subgroup of $\NU{\ZZ G}$ of the 
same order as $G$ is conjugate to $G$ by a unit in $\QQ G$. 
Hence we can assume that $\abs{H}<\abs{G}$. We have $\abs{G}=2^{3}.3.7$. 
The conjugacy classes of $G$ consist of one class each of elements of 
orders $1$, $2$, $4$ and $3$, and two classes of elements of order $7$,
with an element of order $7$ not being conjugate to its inverse.
The Zassenhaus conjecture is valid for $G$, so in particular a torsion
unit in $\NU{\ZZ G}$ has the same order as some element of $G$.

We first show that $H$
is isomorphic to a subgroup of $G$. Recall that $\abs{H}$ divides
$\abs{G}$. So $H$ is solvable as $60$ is not a divisor of $\abs{G}$.
Let $M$ be a minimal normal subgroup of $H$. Then $H$ is an elementary
abelian $p$-group. We assume that $H\neq 1$, so $M\neq 1$.

Suppose that $p=2$. Then $\abs{M}\leq 4$ by Theorem~\ref{t2UG}.
When $\abs{M}=2$, then $H$ is a $2$-group since $H$ contains no 
elements of order $2r$, $r$ an odd prime, and $H$ is isomorphic to
a subgroup of $G$ by Theorem~\ref{t2UG}. Assume that $M$ is a 
four-group. Then $M\neq H$ since $M$ is a minimal normal subgroup of
$H$. Only $2$-elements of $H$ can centralize a nontrivial element of
$M$. It follows that $H$ is isomorphic to either $A_{4}$ or to $S_{4}$.
Both groups occur as subgroups of $G$.

Suppose that $p=3$. Then $\abs{M}=3$. Since $H$ contains no elements
of order $3r$, $r>1$, either $H=M$ or $H$ is isomorphic to $S_{3}$.
So $H$ is isomorphic to a subgroup of $G$.

Finally, suppose that $p=7$. Then $\abs{M}=7$, and $H/M$ acts faithfully
on $M$. So either $H=M$, or $H$ is isomorphic to the Frobenius
group of order $14$, or to the Frobenius group of order $21$.
The latter group occurs as a subgroup of $G$. Suppose that $\abs{H}=14$.
Then $H$ has an element $x$ of order $7$ which is conjugate to its 
inverse. But $x$ is conjugate to an element $g$ of $G$ by a unit in
$\QQ G$, whence $g$ is conjugate to its inverse by a unit in $\QQ G$,
and therefore also by an element in $G$, a contradiction. 
Thus $H$ is isomorphic to a subgroup of $G$.

It remains to prove the conjugacy statement, which will be done using
the suitable (standard) criterion from ordinary character theory (see 
\cite[Lemma~(37.6)]{Seh:93}).
We have to find a subgroup $U$ of $G$ isomorphic to $H$ and an 
isomorphism $\varphi\colon H\rightarrow U$ such that 
$\chi(h)=\chi(\varphi(h))$ for all irreducible characters $\chi$ of $G$
and all $h\in H$. We have shown that there exists an isomorphism 
$\varphi\colon H\rightarrow U$ for some $U\leq G$.
Remember that the Zassenhaus conjecture is valid for $G$.
If $\abs{H}\neq 21$ then, in view of the above 
possibilities for $H$, the isomorphism $\varphi$ has the required 
property. So assume that $\abs{H}=21$. Then we can adjust $\varphi$,
if necessary, by a group automorphism of $H$
so that $\chi(x)=\chi(\varphi(x))$ for an element $x$
of order $7$ in $H$. Then again, $\varphi$ has the required property.
\end{example}

\section{Composition Factors of finite groups of units in 
$\ZZ[\text{\rm PSL}(2,q)]$}\label{sec:CF}

We continue to let $q$ denote a power of a prime $p$. The group 
$\text{\rm PSL}(2,q)$ has, for a prime $r$ distinct from $2$ and $p$,
cyclic Sylow $r$-subgroups. A finite group of units in 
$\ZZ[\text{\rm PSL}(2,q)]$ therefore also has cyclic Sylow $r$-subgroups,
by the following theorem, which is Corollary~1 in \cite{Her:07a}.

\begin{theorem}\label{t1}
Let $G$ be a finite group having cyclic Sylow $r$-subgroups 
for some prime $r$. Then any finite $r$-subgroup of $\NU{\ZZ G}$
is isomorphic to a subgroup of $G$.
\end{theorem}

In the situation of the theorem, a finite subgroup $H$ of $\NU{\ZZ G}$
has cyclic Sylow $r$-subgroups. When $\abs{H}$ is divisible by $r$,
we may obtain more information about the structure of $H$, or even
quotients of $H$ whose order is divisible by $r$, when elements of order
$r$ in $H$ are conjugate to elements of $G$ by units in $\QQ G$, by
comparing the number of conjugacy classes of elements of order $r$ in $H$
(or the quotient of $H$) with the corresponding number of classes in $G$.

We begin with an elementary group-theoretical observation.
For a finite group $G$, we write $\cclr{G}$ for the 
set of the conjugacy classes of elements of order $r$ in $G$.

\begin{lemma}\label{lem1}
Let $H$ be a finite group with cyclic Sylow $r$-subgroups for some prime
divisor $r$ of the order of $H$, and let $X$ be a quotient of $H$
whose order is divisible by $r$. Then $\abs{\cclr{X}}\geq\abs{\cclr{H}}$.
\end{lemma}
\begin{proof}
Let $N$ be the normal subgroup of $H$ with $X=\bar{H}=H/N$, and
let $x$ be a $r$-element in $H$ such that $\bar{x}$ has order $r$.
Note that the subgroup lattice of a Sylow $r$-subgroup of $H$ is
linearly ordered.
By Sylow's theorem, $\lara{\bar{x}}$ contains representatives
of the conjugacy classes of elements of order $r$ in $\bar{H}$, and
$\lara{x}$ is a Sylow $p$-subgroup of $\lara{x}N$.

Suppose that $x$ has order $r$. Then, if $\bar{x}$ is conjugate to
$\bar{x}^{i}$ in $\bar{H}$ for some integer $i$, Sylow's theorem 
implies that $x$ is conjugate to $x^{i}$ in $H$.
So $\abs{\cclr{X}}\geq\abs{\cclr{H}}$ in this case.
Let us therefore assume that the order of $x$ is $r^{n}$, for some
$n>1$. Set $y=x^{r^{n-1}}\in N$. Then $\lara{y}$ contains 
representatives of the conjugacy classes of elements of order $r$ in $H$.
Suppose that $\bar{x}$ is conjugate to $\bar{x}^{i}$ in $\bar{H}$ for
some integer $i$. Let $c\in H$ with $\bar{x}^{\bar{c}}=\bar{x}^{i}$.
By Sylow's theorem, $x^{ca}\in\lara{x}\cap x^{i}N=x^{i}\lara{x^{r}}$
for some $a\in N$. Raising elements to the $r^{n-1}$th power yields
$y^{ca}=y^{i}$. This shows that $\abs{\cclr{X}}\geq\abs{\cclr{H}}$.
\end{proof}

As indicated before, we will use this as follows.

\begin{lemma}\label{lem2}
Let $G$ be a finite group with cyclic Sylow $r$-subgroups for some
prime $r$. Suppose that each unit of order $r$ in $\NU{\ZZ G}$ is 
conjugate to some element of $G$ by a unit in $\QQ G$. Let $X$ be
a quotient of a finite subgroup of $\NU{\ZZ G}$ whose order is 
divisible by $r$.  Then $\abs{\cclr{X}}\geq\abs{\cclr{G}}$.
\end{lemma}
\begin{proof}
A finite subgroup of $\NU{\ZZ G}$ has cyclic Sylow $r$-subgroups,
by Theorem~\ref{t1}. So we can assume, by Lemma~\ref{lem1},
that $X$ is a subgroup $H$ of $\NU{\ZZ G}$. Let $g$ be an element
of order $r$ in $G$. By assumption, each element of order $r$ in $H$
is conjugate to some power of $g$ by a unit in $\QQ G$. Take an element
$h$ in $H$ which is conjugate to $g$ by a unit in $\QQ G$. If $h$ 
is conjugate to $h^{i}$ in $H$ for some integer $i$, then $g$ is 
conjugate to $g^{i}$ by a unit in $\QQ G$, which implies that 
$g$ and $g^{i}$ are conjugate in $G$. Hence 
$\abs{\cclr{H}}\geq\abs{\cclr{G}}$, and the lemma is proven.
\end{proof}

As for the groups $\text{\rm PSL}(2,q)$, the lemma can be applied through
the following proposition, which is Proposition~6.4 in \cite{Her:05b}.

\begin{proposition}\label{prop6.4}
Let $G=\text{\rm PSL}(2,q)$, and let $r$ be a prime distinct from $p$  
(of which $q$ is a power). Then any torsion unit in $\NU{\ZZ G}$ of
order $r$ is conjugate to an element of $G$ by a unit in $\QQ G$.
\end{proposition}

We shall need the following simple number-theoretical lemma.

\begin{lemma}\label{lem3}
Let $a$, $n$ and $m$ be natural integers.
Then $a^{n}-1$ divides $a^{m}-1$ if and only if $n$ divides $m$.
\end{lemma}
 
Now we can prove the following two lemmas. 
The reason for doing so will be given afterwards.
For the first, we use that the number of conjugacy classes of the group 
$\text{PSL}(2,q)$ consisting of elements of order $r$, an odd prime 
divisor of the group order, is $2$ if $r=p$ and $(r-1)/2$ if $r\neq p$.

\begin{lemma}\label{lem4}
Set $G=\text{\rm PSL}(2,q)$ and let $X$ be a composition factor of a finite
subgroup of $\NU{\ZZ G}$. If $X$ is a two-dimensional projective special
linear group, then $X$ is isomorphic to a subgroup of $G$.
\end{lemma}
\begin{proof}
Let $X\cong\text{PSL}(2,r^{m})$ for some prime $r$ and a natural integer
$m$. Suppose that $r=p$. Write $q=p^{f}$. That $\abs{X}$ divides 
$\abs{G}$ means that $r^{m}(r^{2m}-1)$ divides $p^{f}(p^{2f}-1)$, so $m$
divides $f$ by Lemma~\ref{lem3}, and $G$ has a subgroup isomorphic to $X$.
So let us assume that $r\neq p$. When $r=2$, Sylow $2$-subgroups of $X$
are elementary abelian, but also cyclic or dihedral groups, by 
Theorem~\ref{t2UG}. Thus $X\cong\text{PSL}(2,4)\cong A_{5}$. With $5$
dividing $\abs{X}$, also $\abs{G}$ is divisible by $5$, and $G$ has a
subgroup isomorphic to $A_{5}$. So we may assume that $r$ is odd. Then
$X$ has precisely $2$ conjugacy classes of elements of order $r$.
The order of $G$ is divisible by $r$, so $G$ has precisely $(r-1)/2$
conjugacy classes of elements of order $r$. Thus $r\in\{3,5\}$ by
Lemma~\ref{lem2}. Finite $r$-subgroups of $\NU{\ZZ G}$ are cyclic,
by Theorem~\ref{t1}. So $X$ has cyclic Sylow $r$-subgroups, that is,
$m=1$. Since $\text{PSL}(2,3)$ is not simple, it follows that $r=5$ and
$X\cong\text{PSL}(2,5)\cong A_{5}$. Again, $G$ has a subgroup 
isomorphic to $A_{5}$.
\end{proof}

\begin{lemma}\label{lem5}
For $G=\text{\rm PSL}(2,q)$, the alternating group of degree $7$ does not
occur as a composition factor of a finite subgroup of $\NU{\ZZ G}$.
\end{lemma}
\begin{proof}
Suppose, by way of contradiction, that $A_{7}$ is a composition factor
of a finite subgroup of $\NU{\ZZ G}$. Since $A_{7}$ has noncyclic Sylow
$3$-subgroups, Theorem~\ref{t1} shows that $G$ also has noncyclic Sylow
$3$-subgroups. Hence $p=3$. Furthermore, $5$ divides $\abs{G}$, and $G$
has $2$ conjugacy classes of elements of order $5$.
But $A_{7}$ has only one conjugacy class of elements of order $5$.
This contradicts Lemma~\ref{lem2}.
\end{proof}

The main result of this section, which identifies the composition
factors of finite groups of units in $\ZZ[\text{\rm PSL}(2,q)]$, follows
easily if we do not mind using part of the classification of the finite
simple groups. Groups with dihedral Sylow $2$-subgroups were classified by
Gorenstein and Walter in the three papers \cite{GoWa:62,GoWa:64,GoWa:65}
(see \cite[\S\,16.3]{Gor:68}). In particular, if $G$ is a simple group
with dihedral Sylow $2$-subgroups, it is isomorphic to either 
$\text{PSL}(2,q)$, $q$ odd, $q>3$, or to $A_{7}$.
Groups with abelian Sylow $2$-subgroups were classified by Walter
\cite{Wal:69} (a short proof was obtained by Bender \cite{Ben:70}).
In particular, if $G$ is a nonabelian simple group with abelian Sylow 
$2$-subgroups, it is isomorphic to either $\text{PSL}(2,q)$, for certain
$q$, to the Janko group $J_{1}$, or to a Ree group (for a description of
these groups, see \cite[Chapter~XI, \S\,13]{HuBl:82}). 

\begin{theorem}\label{thm6}
For $G=\text{\rm PSL}(2,q)$, each composition factor of a finite subgroup 
of $\NU{\ZZ G}$ is isomorphic to a subgroup of $G$.
\end{theorem}
\begin{proof}
We only have to consider a nonabelian composition factor $X$.
Suppose $p$ (of which $q$ is a power) is odd. 
Then $X$ has dihedral Sylow $2$-subgroups by
Theorem~\ref{t2UG}. By the classification of the finite simple groups
with dihedral Sylow $2$-subgroups, Lemmas~\ref{lem4} and \ref{lem5} show
that $X$ is isomorphic to a subgroup of $G$. So assume that $p=2$.
Then $X$ has abelian Sylow $2$-subgroups, as noted at the beginning of 
\S\,\ref{sec:2groups}. By Lemma~\ref{lem4}, and the 
classification of the finite nonabelian simple groups with abelian Sylow 
$2$-subgroups, we have to consider the possibility that $X$ is isomorphic
to the Janko group $J_{1}$, or to a Ree group. The group $J_{1}$ has
elements of order $7$, all of which are conjugate. So $X$ is not
isomorphic to $J_{1}$ by Lemma~\ref{lem2}. 
A Ree group has a noncyclic subgroup of order $9$ while $G$ has cyclic
Sylow $3$-subgroups. So $X$ is not isomorphic to a Ree group by
Theorem~\ref{t1}. The proof is complete.
\end{proof}

We end with an application of the last theorem.
The classification of all minimal finite simple groups 
(those for which every proper subgroup is solvable) is obtained 
as a corollary of the major work of Thompson on $N$-groups
(consisting of the six papers \cite{Tho:68,Tho:70}).
A minimal finite nonabelian simple group is isomorphic to either 
$\text{PSL}(2,q)$, for certain $q$, or to $\text{PSL}(3,3)$,
or to a Suzuki group $\text{\rm Sz}(2^{p})$ for some odd prime $p$.
  
\begin{theorem}\label{thm7}
Let $G$ be a minimal finite simple group. Then a finite subgroup of
$\NU{\ZZ G}$ of order strictly smaller than $\abs{G}$ is solvable.
\end{theorem}
\begin{proof}
Let $H$ be a nontrivial finite subgroup of $\NU{\ZZ G}$ of order strictly 
smaller than $\abs{G}$. We have to show that $H$ is solvable.
Let $N$ be a maximal normal subgroup of $H$. 
By induction on $\abs{H}$ we can assume that $N$ is solvable, and have 
to show that the simple group $H/N$ is solvable.

If the order of a nonabelian simple group $K$ divides
$\abs{\text{\rm PSL}(3,3)}$, then 
$\abs{K}=\abs{\text{\rm PSL}(3,3)}=2^{4}.3^{3}.13$.
(One even has $K\cong\text{\rm PSL}(3,3)$.)
This is either proved directly or by looking at a list of simple groups
of small order, e.g.\ in \cite{CCNPW:85}.
So we can assume that $G$ is not isomorphic to $\text{\rm PSL}(3,3)$.

If $G$ is isomorphic to $\text{PSL}(2,q)$ for some $q$, solvability
of $H/N$ follows from Theorem~\ref{thm6}.
So by Thompson's work, we can assume that $G$ is isomorphic to a
Suzuki group $\text{\rm Sz}(2^{p})$ for some odd prime $p$.
Then $3$ does not divide $\abs{G}$ (see \cite[XI.3.6]{HuBl:82}),
and so $\abs{H/N}$ is not divisible by $3$. Suppose that $H/N$ is
nonsolvable. The Suzuki groups are the only
finite nonabelian simple groups of order not divisible by $3$
(see \cite[XI.3.7]{HuBl:82}). So $H/N$ is isomorphic to a
Suzuki group $\text{\rm Sz}(2^{n})$ for some odd number $n$, $n\geq 3$.
We have $\abs{G}=2^{2p}(2^{2p}+1)(2^{p}-1)$ and 
$\abs{H/N}=2^{2n}(2^{2n}+1)(2^{n}-1)$. In particular, $2^{n}-1$
divides $\abs{G}(2^{p}+1)$, which equals $2^{2p}(2^{4p}-1)$. 
Thus $n$ divides $4p$ by Lemma~\ref{lem3}, and as $n$ is odd, $n=p$.
So $\abs{H/N}=\abs{G}$, a contradiction.
\end{proof}

Perhaps more could now be said about nonsolvable finite subgroups of 
units in $\ZZ[\text{\rm PSL}(2,q)]$, but we refrain from doing so.
Priority should have the following issues. Let $q$ be a power of the prime
$p$. If the order of a torsion unit in $\NU{\ZZ[\text{\rm PSL}(2,q)]}$
is divisible by $p$, has it order $p$? If $p$ is odd, are units of
order $p$ in $\ZZ[\text{\rm PSL}(2,q)]$ conjugate to elements of 
$\text{\rm PSL}(2,q)$ by units in the rational group ring?
If $p$ is odd, are finite $p$-subgroups of
$\NU{\ZZ[\text{\rm PSL}(2,q)]}$ elementary abelian?

Finally, we remark that \S\,\ref{sec:2groups} can be seen as a
contribution to the determination of the isomorphism types of finite 
$2$-subgroups in $\NU{\ZZ G}$ for finite groups $G$ with dihedral 
Sylow $2$-subgroups (a problem suggested in \cite{Her:07a}).



\begin{thebibliography}{10}

\bibitem{Ben:70}
Helmut Bender, \emph{On groups with abelian {S}ylow {$2$}-subgroups}, Math. Z.
  \textbf{117} (1970), 164--176.

\bibitem{BlHiKi:95}
Frauke~M. Bleher, Gerhard Hiss, and Wolfgang Kimmerle, \emph{Autoequivalences
  of blocks and a conjecture of {Z}assenhaus}, J. Pure Appl. Algebra
  \textbf{103} (1995), no.~1, 23--43.

\bibitem{CoLi:65}
James~A. Cohn and Donald Livingstone, \emph{On the structure of group algebras.
  {I}}, Canad. J. Math. \textbf{17} (1965), 583--593.

\bibitem{CCNPW:85}
  J.~H.~Conway, R.~T.~Curtis, S.~P.~Norton, R.~A.~Parker, and R.~A.~Wilson, 
  \emph{Atlas of finite groups},
  Oxford University Press, Eynsham, 1985.

\bibitem{DoJuMi:97}
Michael~A. Dokuchaev, Stanley~O. Juriaans, and C{\'e}sar Polcino~Milies,
  \emph{Integral group rings of {F}robenius groups and the conjectures of {H}.
  {J}. {Z}assenhaus}, Comm. Algebra \textbf{25} (1997), no.~7, 2311--2325.

\bibitem{Dor:71}
Larry Dornhoff, \emph{Group representation theory. {P}art {A}: {O}rdinary
  representation theory}, Marcel Dekker Inc., New York, 1971, Pure and Applied
  Mathematics, 7.

\bibitem{Gor:68}
Daniel Gorenstein, \emph{Finite groups}, Harper \& Row Publishers, New York,
  1968.

\bibitem{GoWa:62}
Daniel Gorenstein and John~H. Walter, \emph{On finite groups with dihedral
  {S}ylow 2-subgroups}, Illinois J. Math. \textbf{6} (1962), 553--593.

\bibitem{GoWa:64}
\bysame, \emph{On the maximal subgroups of finite simple groups}, J. Algebra
  \textbf{1} (1964), 168--213.

\bibitem{GoWa:65}
\bysame, \emph{The characterization of finite groups with dihedral {S}ylow
  {$2$}-subgroups. {I, II, III}}, J. Algebra \textbf{2} (1965), 85--151;
  218--270; 354--393.

\bibitem{Her:07b}
Martin Hertweck, \emph{The orders of torsion units in integral group rings of
  finite solvable groups}, Comm. Algebra, to appear (e-print {\tt
  arXiv:math/0703541v1 [math.RT]}).

\bibitem{Her:05b}
\bysame, \emph{Partial augmentations and {B}rauer character values of torsion
  units in group rings}, Comm.\ Algebra, to appear (e-print
  \url{arXiv:math.RA/0612429v2}).

\bibitem{Her:07a}
\bysame, \emph{Unit groups of integral finite group rings with no noncyclic
  abelian finite subgroups}, Comm. Algebra, 
  \textbf{36} (2008), no.~9, 3224--3229.

\bibitem{Her:06}
\bysame, \emph{On the torsion units of some integral group rings}, Algebra
  Colloq. \textbf{13} (2006), no.~2, 329--348.

\bibitem{Her:05}
\bysame, \emph{Torsion units in integral group rings of certain metabelian
  groups}, Proc. Edinb. Math. Soc. \textbf{51} (2008), no.~2, 363--385.

\bibitem{Hoef:08}
Christian~R. H\"ofert, \emph{{Bestimmung von Kompositionsfaktoren endlicher
  Gruppen aus Burnside\-ringen und ganzzahligen Gruppenringen}}, Ph.D. thesis,
  University of Stuttgart, 2008, online publication,
  \url{http://elib.uni-stuttgart.de/opus/volltexte/2008/3541/}.

\bibitem{HuBl:82}
Bertram Huppert and Norman Blackburn, \emph{Finite groups. {III}}, Grundlehren
  der Mathemati\-schen Wissenschaften [Fundamental Principles of Mathematical
  Sciences], vol. 243, Springer-Verlag, Berlin, 1982.

\bibitem{Jan:74}
G.~J. Janusz, \emph{Simple components of {$Q[{\rm SL}(2,\,q)]$}}, Comm. Algebra
  \textbf{1} (1974), 1--22.

\bibitem{OWR:07}
E.~Jespers, Z.~Marciniak, G.~Nebe, and W.~Kimmerle, \emph{Oberwolfach
  {R}eports. {V}ol. 4, no. 4}, Report {N}o.~55/2007 {M}ini-{W}orkshop:
  {A}rithmetik von {G}ruppenringen, pp.~3149--3179, European Mathematical
  Society (EMS), Z\"urich, 2007.

\bibitem{Ki:07}
Wolfgang Kimmerle, 
\emph{Torsion units in integral group rings of finite insoluble groups},
\cite[pp.~3169-3170]{OWR:07}.

\bibitem{KiLySaTe:90}
Wolfgang Kimmerle, Richard Lyons, Robert Sandling, and David~N. Teague,
  \emph{Composition factors from the group ring and {A}rtin's theorem on orders
  of simple groups}, Proc. London Math. Soc. (3) \textbf{60} (1990), no.~1,
  89--122.

\bibitem{San:81}
Robert Sandling, \emph{Graham {H}igman's thesis ``{U}nits in group rings''},
  Integral representations and applications (Oberwolfach, 1980), Lecture Notes
  in Math., vol. 882, Springer, Berlin, 1981, pp.~93--116.

\bibitem{Seh:93}
S.~K. Sehgal, \emph{Units in integral group rings}, Pitman Monographs and
  Surveys in Pure and Applied Mathematics, vol.~69, Longman Scientific \&
  Technical, Harlow, 1993.

\bibitem{Seh:03}
\bysame, \emph{Group rings}, Handbook of algebra, Vol. 3, North-Holland,
  Amsterdam, 2003, pp.~455--541.

\bibitem{Sha:83}
M.~A. Shahabi~Shojaei, \emph{Schur indices of irreducible characters of {${\rm
  SL}(2,\,q)$}}, Arch. Math. (Basel) \textbf{40} (1983), no.~3, 221--231.

\bibitem{Suz:82}
Michio Suzuki, \emph{Group theory. {I}}, Grundlehren der Mathematischen
  Wissenschaften [Fundamental Principles of Mathematical Sciences], vol. 247,
  Springer-Verlag, Berlin, 1982, Translated from the Japanese by the author.

\bibitem{Tho:68}
John~G. Thompson, \emph{Nonsolvable finite groups all of whose local subgroups
  are solvable}, Bull. Amer. Math. Soc. \textbf{74} (1968), 383--437.

\bibitem{Tho:70}
\bysame, \emph{Nonsolvable finite groups all of whose local subgroups are
  solvable. {II}, {III}, {IV}, {V}, {VI}}, Pacific J. Math. \textbf{33} (1970),
  451--536, ibid. 39 (1971), 483--534; ibid. 48 (1973), 511--592; ibid. 50
  (1974), 215--297; ibid. 51(1974), 573--630.

\bibitem{Wal:69}
John~H. Walter, \emph{The characterization of finite groups with abelian
  {S}ylow {$2$}-subgroups.}, Ann. of Math. (2) \textbf{89} (1969), 405--514.

\bibitem{Wei:91}
Alfred Weiss, \emph{Torsion units in integral group rings}, J. Reine Angew.
  Math. \textbf{415} (1991), 175--187.

\end{thebibliography}

\providecommand{\bysame}{\leavevmode\hbox to3em{\hrulefill}\thinspace}

\end{document}